\documentclass[11pt,reqno]{amsart}
\usepackage{amssymb,amsmath,amscd}
\usepackage{latexsym,bm,bbm,mathrsfs}
\usepackage{hyperref,graphicx, enumerate}
\usepackage{mathtools}
\usepackage{stmaryrd}
\usepackage[all,pdf]{xy}
\usepackage{graphicx}
\usepackage{caption}
\usepackage{tikz-cd}
\usepackage{color}
\graphicspath{ {./} }

\usepackage{stackrel} 
\usepackage[left=30mm, right=30mm, top=30mm, bottom=30mm]{geometry}
\usepackage[shortlabels]{enumitem}
\usepackage{ulem}
\usepackage{verbatim}

\allowdisplaybreaks

\newtheorem{thm}{Theorem}[section]

\newtheorem*{quest*}{Question}

\newtheorem{theorem}[thm]{Theorem}

\newtheorem{lemma}[thm]{Lemma}

\newtheorem{remark}[thm]{Remark}
\newtheorem{defn}[thm]{Definition}

\newtheorem*{theorem*}{Theorem}

\newcommand{\tors}{\left[\operatorname{tors}\right]}

\newcommand{\bbf}{{\mathbb{F}}}

\newcommand{\bbq}{{\mathbb{Q}}}

\newcommand{\bbz}{{\mathbb{Z}}}

\newcommand{\cB}{\mathcal{B}}

\newcommand{\GLp}{\operatorname{GL}_{2}\left(\Fp\right)}

\newcommand{\Gal}{\operatorname{Gal}}
\renewcommand{\Im}{\operatorname{Im}}

\newcommand{\Fp}{\bbf_{p}}
\newcommand{\Fpp}{\bbf_{p^{2}}}
\newcommand{\Fpx}{\bbf_{p}^{\times}}
\newcommand{\Fppx}{\bbf_{p^{2}}^{\times}}
\newcommand{\cC}{\mathcal{C}}

\newcommand{\End}{\operatorname{End}}
\newcommand{\cyc}{\operatorname{cyc}_{p}}

\title[]{Uniform bounds on prime levels of abelian division fields of elliptic curves over number fields without rationally defined CM}

\author{Hansol Kim}
\address{Department of Mathematics, Ajou University, Suwon 16499, Republic of Korea}
\email{jawlang@ajou.ac.kr}
\thanks{Hansol Kim was supported by the National Research Foundation of Korea (NRF) grant funded by the Korea government (MSIT) (No. RS-2024-00334558).}

\date{\today}
\subjclass[2020]{Primary 11G05,  Secondary 14H52, 11F80}
\keywords{elliptic curve, Galois representation}

\begin{document}
\maketitle
\begin{abstract}
	Enrique Gonz{\'a}lez-Jim{\'e}nez and {\'A}lvaro Lozano-Robledo proved that there exists a uniform bound on the levels $n$ for which the $n$-division field $\bbq\left(E\left[n\right]\right) / \bbq$ of an elliptic curve $E/\bbq$ defined over $\bbq$ is abelian. Assuming GRH (Generalized Riemann Hypothesis), Allen and Genao partially generalized this result by restricting to prime levels while allowing arbitrary number fields without rationally defined CM. In this paper, we remove the GRH assumption and prove that the property established by Allen and Genao is in fact equivalent to the absence of rationally defined CM.
\end{abstract}

\section{Introduction}

We let $E/K$ be an elliptic curve defined over a number field $K$. By Serre \cite{Serre72}, if $E$ does not have CM, then the mod-$p$ Galois representation is surjective for all sufficiently large primes $p$. In contrast, if $E$ has CM, then the image of the mod-$p$ Galois representation is contained in the normalizer of a Cartan subgroup for all sufficiently large primes $p$ (see, e.g., \cite[Ch.~II]{Silverman2}; see also \cite[§~7]{Z15} for an explicit argument showing the containment in the normalizer of a Cartan subgroup over $\bbq$).

Motivated by this dichotomy, several uniformity results have been established over number fields without RCM (rationally defined CM).

\begin{defn}[{\cite[p.2]{G24}}]
	We say a number field $K$ does not have RCM if $\End_{K}\left(E\right) = \bbz$ for any elliptic curve $E/K$ defined over $K$.
\end{defn}

\begin{theorem}[{\cite[Theorems~1~and~3]{G24}}]\label{thm:growth}
	If $K$ has the following {\normalfont \textbf{Property}}~\eqref{eqn:property_growth}\begin{equation}\label{eqn:property_growth}
		\begin{aligned}
			&\text{there exists a constant $C_{K}$ depending only on $K$ such that}\\
			&\text{for any elliptic curve } E/K \text{ and }\\
			&\text{any finite extension } L/K \text{ of degree whose minimal prime divisor is } > C_{K}\\
			&E\left(L\right)\tors = E\left(K\right)\tors,
		\end{aligned}
	\end{equation} then $K$ does not have RCM. The converse holds under GRH (Generalized Riemann Hypothesis).
\end{theorem}

Without GRH, the converse part of Theorem~\ref{thm:growth} is proven for $K$ such that $\left[K:\bbq\right] = 1,2$ by \cite[Theorem~7.2.i]{GN20} and \cite[Theorem~1.2]{IK24}, respectively.

\begin{theorem}[{\cite[Theorem~7.2.i]{GN20}}]
	$\bbq$ has \textbf{{\normalfont \textbf{Property}}}~\eqref{eqn:property_growth}.
\end{theorem}

\begin{theorem}[{\cite[Theorem~1.2]{IK24}}]
	A quadratic number field $K$ has {\normalfont \textbf{Property}}~\eqref{eqn:property_growth} if and only if $K$ does not have RCM.
\end{theorem}

For a positive integer $n$, we denote the $n$-torsion subgroup $E\left(\overline{K}\right)\left[n\right]$ simply by $E\left[n\right]$. The field extension $K\left(E\left[n\right]\right) / K$ is called the $n$-division field extension of $E/K$.
In \cite{GJLR16}, it is shown that the $n$-division field extension can be abelian only for uniformly bounded values of $n$.

\begin{theorem}[{\cite[Theorem~1.1]{GJLR16}}]\label{thm:GJLR16}
	We let $E/\bbq$ be an elliptic curve defined over $\bbq$, $n \ge 2$ a positive integer, and $\zeta_{n}$ a primitive $n$th root of unity. \begin{itemize}
		\item If $\bbq\left(E\left[n\right]\right) = \bbq\left(\zeta_{n}\right)$, then $n \in \left\{2,3,4,5\right\}$.
		\item If $\Gal\left(\bbq\left(E\left[n\right]\right)/\bbq\right)$ is abelian, then $n \in \left\{2,3,4,5,6,8\right\}$.
	\end{itemize}
\end{theorem}

Allen and Genao partially generalized \cite[Theorem~1]{GJLR16} by proving the following result concerning the prime $p$-division field extension $K\left(E\left[p\right]\right)$ over a number field $K$ without RCM.

\begin{theorem}[{\cite[Theorem~3]{AG25}}]\label{thm:AG3_2}
	Assuming GRH, a number field $K$ without RCM has the following property: \begin{equation}\label{eqn:property}
		\begin{aligned}
			&\text{there exists a constant } C_{K}>0 \text{ depending only on } K \text{ such that,}\\
			&\text{for any elliptic curve } E/K \text{ defined over } K \text{ and any prime } p > C_{K}\\
			&\Gal\left(K\left(E\left[p\right]\right)/K\right) \text{is not abelian.}
		\end{aligned}
	\end{equation}
\end{theorem}

The main purpose of this paper is to prove, without assuming GRH, the equivalence between {\normalfont \textbf{Property}}~\eqref{eqn:property} and the absence of RCM.

\begin{theorem}[{}]\label{thm:main}
	$K$ has {\normalfont \textbf{Property}}~\eqref{eqn:property} if and only if $K$ does not have RCM.
\end{theorem}

As an incidental consequence of Lemma~\ref{lem:absub_not_Cartan}, which is a key ingredient in the proof of Theorem~\ref{thm:main}, we also obtain Theorem~\ref{thm:Cartan}. Theorem~\ref{thm:Cartan} can be viewed as complementary to \cite[Theorem~3]{AG25}: while both results have a similar structure, they are optimized for different classes of base number fields. We also recall another part of \cite[Theorem~3]{AG25} for comparison with Theorem~\ref{thm:Cartan}.

\begin{theorem}[{\cite[Theorem~3]{AG25}}]\label{thm:AG3_1}
	We let $E/K$ be an elliptic curve defined over a number field $K$ and $\Delta_{K}$ the discriminant of the extension $K/\bbq$. For a prime $p \ge 17$ satisfying either $p \nmid \Delta_{K}$ or $p > \left(6\left[K:\bbq\right]\right)!+ 1 $, if $\Gal\left(K\left(E\left[p\right]\right)/K\right)$ is abelian, then the image of the mod-$p$ Galois representation is contained in a Cartan subgroup of $\GLp$.
\end{theorem}

\begin{theorem}[{}]\label{thm:Cartan}
	We let $K$ be a number field and $p \ge 3$ a prime satisfying $\sqrt{\left(-1\right)^{\frac{p-1}{2}}p} \notin K$. For any elliptic curve $E/K$ defined over $K$, if $\Gal\left(K\left(E\left[p\right]\right)/K\right)$ is abelian, then the image of the mod-$p$ Galois representation is contained in a Cartan subgroup of $\GLp$.
\end{theorem}

\begin{remark}
	Compared with \cite[Theorem~3]{AG25}, Theorem~\ref{thm:Cartan} yields a smaller lower bound on $p$ for certain number fields $K$ under which, if the Galois group $\Gal\left(K\left(E\left[p\right]\right) / K\right)$ is abelian, then the image of the mod-$p$ Galois representation is contained in a Cartan subgroup. We now illustrate this comparison with concrete examples.
	
	Assume that the maximal polyquadratic subextension of $K/\bbq$ is $\bbq$. By Theorem~\ref{thm:Cartan}, for every odd prime $p$, if the Galois group $\Gal\left(K\left(E\left[p\right]\right) / K\right)$ is abelian, then the image of the mod-$p$ Galois representation is contained in a Cartan subgroup. By contrast, \cite[Theorem~3]{AG25} requires $p \ge 17$. In particular, if $\left[K:\bbq\right]$ is odd and $\Gal\left(K\left(E\left[p\right]\right) / K\right)$ is abelian, then the image of the mod-$p$ Galois representation is contained in a Cartan subgroup for every odd prime $p$.
	
	We now present another example illustrating this difference. Let $K = \bbq\left(\sqrt{D}\right)$, where $D={\displaystyle\prod_{\text{ prime } \ell \le 12!+1}}\ell$. Again, Theorem~\ref{thm:Cartan} yields the same conclusion for every odd prime $p$, whereas \cite[Theorem~3]{AG25} proves it only for $p>12!+1$.
\end{remark}

In Section~\ref{sec:background}, we review mod-$p$ Galois representations attached to elliptic curves, recall the definition of Cartan subgroups of $\GLp$, and classify the abelian subgroups of $\GLp$. In Section~\ref{sec:proof}, we prove the main theorems~\ref{thm:main}~and~\ref{thm:Cartan}.

\section{Galois representations and abelian subgroups in $\GLp$}\label{sec:background}

\subsection{Mod-$p$ Galois representations attached to an elliptic curve}\label{sec:Gal_repn}

Throughout this paper, we let $K$ be a number field, $E/K$ an elliptic curve defined over $K$, $\Gamma_{K}$ the absolute Galois group of $K$, $p\ge3$ a prime, and $I_{2} \in \GLp$ the identity matrix.

Viewing $E\left[p\right]$ as an $\Fp$-representation of $\Gamma_{K}$, a choice of basis $\cB = \left\{P,Q\right\}$ of the $\Fp$-vector space $E\left[p\right]$ yields the mod-$p$ Galois representation $\rho_{\cB}: \Gamma_{K} \to \GLp$ attached to $E/K$ by $$
	\begin{pmatrix}P^{\sigma}\\Q^{\sigma}\end{pmatrix}
	=
	\rho_{\cB}\left(\sigma\right)
	\begin{pmatrix}P\\Q\end{pmatrix}
$$ for $\sigma \in \Gamma_{K}$.

Via Weil pairing (\cite[Ch.III.Proposition 8.1]{Silverman}), we have that \begin{equation}\label{eqn:Weil}
	\cyc = \det \circ \rho_{\cB}
\end{equation} where $\cyc$ is the mod-$p$ cyclotomic character of $\Gamma_{K}$. For a choice of basis $\cB$, we denote the image of the Galois representation $\rho_{\cB}$ by $G_{E/K,p}$. The statements concerning $G_{E/K,p}$ in Section~\ref{sec:proof} will be independent of the choice of $\cB$.

\subsection{Abelian subgroups in $\GLp$}\label{sec:GLp}

To classify the abelian subgroups in $\GLp$, we recall the standard definitions and basic properties of Cartan subgroups in $\GLp$ (see, for example, \cite[§26.2~Corollary~A~(b)~and~p.125]{H}). A Cartan subgroup is a maximal diagonalizable subgroup over a fixed algebraic closure $\overline{\Fp}$ of $\Fp$. Every Cartan subgroup of $\GLp$ is conjugate over $\Fp$ to exactly one of the following two subgroups $$
	\cC_{s} := \left\{ \begin{psmallmatrix}a&0\\0&d\end{psmallmatrix}: a,d \in \Fpx\right\} \text{ or }
	\cC_{ns} := \left\{ \begin{psmallmatrix}a&b\alpha\\b&a\end{psmallmatrix}: a,b\in \Fp,\left(a,b\right) \ne \left(0,0\right)\right\},
$$ where $\alpha$ is a non-square in $\Fpx$. Cartan subgroups that are conjugate to $\cC_{s}$ are called split. Cartan subgroups that are conjugate to $\cC_{ns}$ are called non-split. By definition, a matrix in a split Cartan subgroup has eigenvalues in $\Fpx$. After extending scalars from $\Fp$ to $\Fpp$, \begin{equation}\label{eqn:Cns}
	\text{ a non-split Cartan subgroup is conjugate to } \left\{ \begin{psmallmatrix}\gamma&0\\0&\gamma^{p}\end{psmallmatrix}: \gamma \in \Fppx\right\}.
\end{equation}

However, a non-split Cartan subgroup does not split over $\Fp$. In other words, a non-split Cartan subgroup is not diagonalizable over $\Fp$.

We classify abelian subgroups in $\GLp$ into two types.

\begin{lemma}\label{lem:absub_not_Cartan}
	We let $p \ge 3$ be a prime and $G \subseteq \GLp$ an abelian subgroup. Then, either \begin{enumerate}[{\normalfont (i)}]
		\item\label{it:ab_Cartan} $G$ is contained in a Cartan subgroup or
		\item\label{it:ab_Borel} $G = \left\{ aU^{b}: a\in H,b\in\Fp \right\}$ for a subgroup $H \subseteq \Fpx$ and a matrix $U$ of order $p$.
	\end{enumerate}
	If \ref{it:ab_Borel} holds, then $\det G = \left\{a^{2}: a \in H\right\}$.
\end{lemma}

\begin{proof}
	For a matrix $R \in \GLp$ with a repeated eigenvalue $a \in \Fpx$, the result follows from the characteristic polynomial of $R$ and induction on $n \ge 0$: \begin{equation}\label{eqn:induction}
		R^{n} = n a^{n-1} R + \left(1-n\right) a^{n} I_{2}.
	\end{equation} Substituting $n$ by the order $\left|R\right|$ of $R$ into \eqref{eqn:induction}, if $p \nmid \left|R\right|$, then $R = c I_{2}$ for some $c \in \Fp$. Since $R$ has a repeated eigenvalue $a$,
	\begin{equation}\label{eqn:scalar}
		\text{if } p \nmid \left|R\right|, \text{ then } R = a I_{2}.
	\end{equation} Substituting $n= p$ into \eqref{eqn:induction}, \begin{equation}\label{eqn:p_pow}
		R^{p} = a I_{2}.
	\end{equation}
	
	
	We assume $p \nmid \left|G\right|$ and prove \ref{it:ab_Cartan}. Recalling the definition of Cartan subgroups, since $G$ is abelian, it suffices to show that each matrix $A \in G$ is diagonalizable. If $A$ has distinct eigenvalues, there is nothing to prove. If $A$ has a repeated eigenvalue $a$, then by \eqref{eqn:scalar} and $p\nmid \left|G\right|$, we have $A=aI_{2}$.
	
	We assume $p \mid \left|G\right|$ and prove \ref{it:ab_Borel}. By Cauchy's theorem, $G$ contains a matrix $U$ of order $p$. For $G' := \left\langle A^{p}, U: A \in G\right\rangle$, we show that $G = G'$. Obviously, $G \supseteq G'$. Conversely, for each $A \in G$, we show that $A \in G'$. We let $k := \frac{\left|A\right|}{\gcd\left(p,\left|A\right|\right)}$. Since $G$ is abelian and $p^{2} \nmid \left|\GLp\right|$, $\left\langle U \right\rangle$ is the unique Sylow $p$-subgroup and $\gcd\left(k,p\right) = 1$. Since the order $\gcd\left(p,\left|A\right|\right)$ of $A^{k}$ is either $1$ or $p$, we have $A^{k} \in \left\langle U \right\rangle$ and $A \in \left\langle A^{p}, A^{k} \right\rangle \subseteq \left\langle A^{p}, U \right\rangle \subseteq G'$. Since $G$ is abelian, $H := \left\{a \in \Fpx: aI_{2} = A^{p} \text{ for some } A\in G\right\} \subseteq \Fpx$ is a subgroup. We prove that $G' = \left\{ aU^{b}: a\in H,b\in\Fp \right\}$. It suffices to show that, for all $A \in G$, there exists $a \in \Fpx$ such that $A^{p} = a I_{2}$. If both of $A$ and $AU$ have distinct eigenvalues in $\Fppx$, then $A^{p^{2}} = A$ and $A^{p^{2}} = A^{p^{2}} U^{p^{2}} = \left(AU\right)^{p^{2}} = AU$. This contradicts that $U \ne I_{2}$. Hence, at least one of $A$ and $AU$ has a repeated eigenvalue. Since $\left(AU\right)^{p} = A^{p}U^{p} = A^{p}$, by \eqref{eqn:p_pow}, $A^{p} = aI_{2}$ for some $a \in \Fpx$.
	
	If \ref{it:ab_Borel} holds, then $U^{p} = I_{2}$ and $\left(\det U\right)^{p} =1$. Hence, $\det U = 1$ and $\det G = \left\{a^{2}: a \in H\right\}$.
\end{proof}

\section{Proofs of main theorems}\label{sec:proof}

Lemma~\ref{lem:large_p} will be a key ingredient in the proof of Theorem~\ref{thm:main}. For convenience, throughout Section~\ref{sec:proof}, we fix a generator $\alpha$ of $\Fpx$.

\begin{lemma}\label{lem:large_p}
	We let $E/K$ be an elliptic curve defined over $K$ and $p \ge 7$ a unramified prime in $K$. If $G_{E/K,p}$ is contained in a split Cartan subgroup, then $G_{E/K,p}$ contains a matrix with eigenvalues $1$ and $\alpha^{e}$ for some $e\in \left\{1,2,3,4,6\right\}$.
\end{lemma}
\begin{proof}
	By \cite[Theorem~3.1]{LR13} and \eqref{eqn:Cns}, for some $e\in \left\{1,2,3,4,6\right\}$, $G_{E/K,p}$ contains either a matrix with eigenvalues $1$ and $\alpha^{e}$ or a matrix with an eigenvalue $\beta^{e}$ where $\beta$ is a generator of $\Fppx$. We suppose that $G_{E/K,p}$ contains either a matrix with an eigenvalue $\beta^{e}$ and lead a contradiction. Since $G$ is contained in a split Cartan subgroup, $\beta^{e} \in \Fpx = \left\langle \beta^{p+1} \right\rangle$ and $p+1$ divides $e$. This contradicts that $p \ge 7$.
\end{proof}

We prove Theorems~\ref{thm:main}.
\begin{proof}[Proof of Theorem~\ref{thm:main}]
	If there exists an elliptic curve $E/K$ defined over $K$ such that $\End_{K}\left(E\right) \supsetneq \bbz$, then by \cite[Ch.III.Exercise~3.24]{Silverman}, for any prime $p$, $\Gal\left(K\left(E\left[p\right]\right)/K\right)$ is abelian. Hence, $K$ does not satisfy {\normalfont \textbf{Property}}~\eqref{eqn:property}.
	
	Now, we show that $K$ without RCM satisfies {\normalfont \textbf{Property}}~\eqref{eqn:property}. We let $S_{K}$ be a finite set of primes depending only on $K$, as given in \cite[Theorem~1]{LV14}. There exists a positive constant $C_{K} $ depending only on $K$ such that any prime $p > C_{K}$ satisfies the following four conditions: \begin{enumerate}[{(A)}]
		\item\label{it:cyc} $\Im \cyc = \Fpx$,
		\item\label{it:SK} $p \notin S_{K}$,
		\item\label{it:unr} $p$ is unramified in $K$, and
		\item\label{it:large} $p \ge 41$.
	\end{enumerate}
	
	We suppose that there exist a prime $p >C_{K}$ and an elliptic curve $E/K$ defined over $K$ such that $\Gal\left(K\left(E\left[p\right]\right)/K\right)$ is abelian and lead a contradiction.

	By Lemma~\ref{lem:absub_not_Cartan}, \eqref{eqn:Weil}, and Condition~\ref{it:cyc}, we obtain that $G_{E/K,p}$ is contained in a Cartan subgroup.
	
	By the definition of Cartan subgroups, the $\Fp$-representation $E\left[p\right]$ of the absolute Galois group $\Gamma_{K}$ of $K$ is reducible over $\overline{\Fp}$. By Condition~\ref{it:SK} and \cite[Theorem~1]{LV14}, either \cite[Theorem~1~(1)]{LV14} or \cite[Theorem~1~(2)]{LV14} holds. Since $K$ does not have RCM, \cite[Theorem~1~(1)]{LV14} is ruled out. Therefore, \cite[Theorem~1~(2)]{LV14} holds.\label{key}

	By \cite[Theorem~1~(2)]{LV14}, $E\left[p\right]$ is reducible (over $\Fp$). Since $G_{E/K,p}$ is contained in a Cartan subgroup, it is contained in a split Cartan subgroup. In other words, there exist a basis $\cB$ of $E\left[p\right]$ and characters $\psi_{1}$ and $\psi_{2}$ from $\Gamma_{K}$ to $\Fpx$ such that $\rho_{\cB} = \begin{psmallmatrix}\psi_{1}&0\\0&\psi_{2}\end{psmallmatrix}$. By \cite[Theorem~1~(2)]{LV14} and \eqref{eqn:Weil}, $\psi_{1}^{6} = \psi_{2}^{6}$. By Condition~\ref{it:unr} and Lemma~\ref{lem:large_p}, $\Im \rho_{\cB}$ contains a matrix with eigenvalues $1$ and $\alpha^{e}$. Hence, $6e \equiv 0 \pmod{p-1}$. Since $e\in\left\{1,2,3,4,6\right\}$, this contradicts the Condition~\ref{it:large}.

\end{proof}

Theorem~\ref{thm:Cartan} is an incidental consequence of Lemma~\ref{lem:absub_not_Cartan}.
\begin{proof}[Proof of Theorem~\ref{thm:Cartan}]
	By Lemma~\ref{lem:absub_not_Cartan} and \eqref{eqn:Weil}, if $\Gal\left(K\left(E\left[p\right]\right)/K\right)$ is abelian but $G_{E/K,p}$ is not contained in a Cartan subgroup, then $\Im \cyc \subseteq \left\{a^{2}:a\in\Fpx\right\}$. This contradicts that $\sqrt{\left(-1\right)^{\frac{p-1}{2}}p} \notin K$.
\end{proof}

\end{document}